\documentclass[11pt,en,cite=numbers]{elegantpaper}

\title{Indices of equilibrium points of linear control systems with saturated state feedback}
\author{Xiao-Song Yang\footnote{1. School of Mathematics and Statistics, Huazhong University of Science and Technology, Wuhan 430074, Hubei, People’s Republic of China. 2. Hubei Key Laboratory of Engineering Modeling and Science Computing, Huazhong University of Science and Technology, Wuhan 430074, Hubei, People’s Republic of China. Email:yangxs@hust.edu.cn.} \: and Weisheng Huang\footnote{School of Mathematics and Statistics, Huazhong University of Science and Technology, Wuhan 430074, Hubei, People’s Republic of China. Email:huangws\_18@hust.edu.cn.}}

\date{}

\begin{document}

\maketitle

\begin{abstract}
In this paper we investigate some properties of equilibrium points in n-dimensional linear control systems with saturated state feedback. We provide an index formula for equilibrium points and discuss its relation to boundaries of attraction basins in feedback systems with single input. In addition, we also touch upon convexity of attraction basin.\par
\end{abstract}

\section{Introduction}
Attraction basin of an attractor is a domain in which every point has the property that a trajectory starting at it approaches the attractor as time goes to infinity. Clearly, attraction basin is a central research focus in dynamical system theory and control theory because of its practical significance as well as a theoretical challenge.

One topic of much importance is the structure of the boundary of the attraction basin of a locally stable equilibrium point. To some extent, the structure of the basin boundary provides a measure of how a trajectory in the attraction basin approaches the equilibrium point. For example, fractal boundary may give rise to transient chaotic behaviour for trajectories near the boundary. Thus a perturbed state probably goes through a long period irregular motion before going to the stable equilibrium state!

For linear systems with saturated state feedback, the topic on the structure of the basin boundary has received much attention in recent decades \cite{hu2001control,Kapila2002Actuator,2011Stability,Corradini2012Control,Li2018stability,Rodolfo1997Linear,Hu2002Convex}. 

Among these is the elegant and complete treatment \cite{hu2001control} on the boundary of the attraction basin (the domain of stability as termed as in \cite{hu2001control}) and on convexity of the attraction basin of the origin in two dimensional situation. Now it is well known that the boundary of the attraction basin is a convex differentiable closed curve, which is a limit cycle of closed loop system. In higher dimensional cases, even though there are many publications on estimation of attraction basin in linear control system with stabilizing saturated state feedback, no such satisfactory results have been obtained up to now for anti-stable linear systems, to the best knowledge of the authors. It is well-known that even in 3-dimensional case, studying dynamics of a nonlinear system is nearly a formidable task, and this is also the case in studying the boundaries of attraction basins of attractors.

Since existence and distribution of equilibrium points affect estimation of attraction basin, we will investigate properties of equilibrium points in n-dimensional linear control systems with saturated state feedback and mainly focus on the single input case. In addition, we also touch upon convexity of attraction basin.

\section{Indices of differentiable maps}
Consider a differentiable map $f: \mathbb{R}^n \to \mathbb{R}^n$. Suppose that there is an $r>0$ such that $f(x) \neq 0$ for all $x \in S^{n-1}(r)$, where
\begin{equation*}
	S^{n-1}(r) = \{ x \in \mathbb{R}^n \vert \, \lVert x \rVert = r \}.
\end{equation*}

Define the sphere map $\bar{f}: S^{n-1}(r) \to S^{n-1}(1)$ as
\begin{equation*}
	\bar{f}(x) = \frac{f(x)}{\lVert f(x) \rVert}, \; x \in S^{n-1}(r).
\end{equation*}

\begin{definition}
	The index of $f$, denoted by $\text{ind} \, f_{B(r)}$, on $B(r)$ is defined as the topological degree of $\bar{f}$ where
\begin{equation*}
	B(r) = \{ x \in \mathbb{R}^n \vert \, \lVert x \rVert \leq r \}.
\end{equation*}
\end{definition}

For degree of differentiable maps, see \cite{1965Topology}, for degree of continuous map, the reader is referred to \cite{Brown1993A}.

In the following we give a simple result which is useful for the arguments in section 3.

\begin{proposition}
	Consider a differentiable map $F: \mathbb{R}^n \to \mathbb{R}^n$,
	\begin{equation*}
		F(x) = f(x) + g(x),\; x \in \mathbb{R}^n.
	\end{equation*}
	Suppose that there is an $r>0$ such that
	\begin{equation*}
		\lVert f(x) \rVert > \lVert g(x) \rVert,\; x \in S^{n-1}(r).
	\end{equation*}
	Then
	\begin{equation*}
		\textnormal{ind} \, F_{B(r)} = \textnormal{ind} \, f_{B(r)}.
	\end{equation*}
\end{proposition}

The proof of this statement is an exercise in differential topology. We provide a proof for reader's convenience.

\begin{proof}
	Consider the homotopy $\hat{F}: [0,1] \times S^{n-1}(r) \to S^{n-1}(1)$,
	\begin{equation*}
		\hat{F}(t,x) = \frac{f(x) + tg(x)}{\lVert f(x) + tg(x) \rVert}, \; (t,x) \in [0,1] \times S^{n-1}(r).
	\end{equation*}
	Then $\hat{F}$ makes sense, and is a homotopy between $F$ and $f$:
	\begin{equation*}
		\hat{F}(0,x) = \bar{f}(x) = \frac{f(x)}{\lVert f(x) \rVert}, \; \hat{F}(1,x) = \bar{F}(x).
	\end{equation*}
	Since the index is homotopy invariant, then $\textnormal{ind} \, F_{B(r)} = \textnormal{ind} \, f_{B(r)}$.
\end{proof}

In particular, we have the following fact.

\begin{proposition}\label{pro:2-2}
	Let $A$ be a nonsingular matrix. Suppose $g: \mathbb{R}^n \to \mathbb{R}^n$ is bounded, then there is an $r>0$ such that the map $Ax + g(x)$ restricted to $B(r)$, has index $(-1)^m$, where $m$ is the number of eigenvalues of $A$ with negative real parts.
\end{proposition}

This fact is obvious, because for the map $\bar{A} = \frac{Ax}{\lVert Ax \rVert}$, one has ind\,$\bar{A}_{B(r)} = \text{sign}(\det A)$ by arguments in \cite{1965Topology}.

For a continuous map $f: \mathbb{R}^n \to \mathbb{R}^n$. Suppose that $f$ has only isolated zero points. Let $\bar{x}$ be a zero point of $f$, the index of at $\bar{x}$ is defined as
\begin{equation*}
	\textnormal{ind} \, f(\bar{x}) = \text{degree of} \; \bar{f}: S^{n-1}(\bar{x},\delta) \to S^{n-1}(1)
\end{equation*}
where
\begin{equation*}
	S^{n-1}(\bar{x}, \delta) = \{ x \in \mathbb{R}^n \vert \, \lVert x - \bar{x} \rVert = \delta \}
\end{equation*}
and
\begin{equation*}
	\bar{f}(x) = \frac{f(x)}{\lVert f(x) \rVert}, \; x \in S^{n-1}(x,\delta).
\end{equation*}
We have the following theorem.

\begin{theorem}\label{th:2-3}
	Suppose $f: B(r) \to \mathbb{R}^n$ has the following properties:
	\begin{enumerate}[1)]
		\item $\lVert f(x) \rVert \neq 0$ for $\forall x \in S^{n-1}(r)$.
		\item Every zero point of $f$ is isolated.
	\end{enumerate}
	Denote by $E$ the set of zero points in $B(r)$, then
	\begin{equation*}
		\sum_{x \in E} \textnormal{ind}_f \, (\bar{x}) = \textnormal{ind} \, f_{B(r)}.
	\end{equation*}
\end{theorem}
The proof is an elementary exercise in differential topology.

For reader's convenience, we will give a proof for differentiable maps, because continuous maps can be approximated by differentiable maps \cite{Brown1993A}. To prove this theorem, we need the following lemma which is adapted from \cite{1965Topology}.

\begin{lemma}
	Let $M$ be a compact oriented manifold, and $K = \partial M$ be the boundary of $M$. $N$ is a connected differentiable manifold. Suppose $\text{dim} K = \text{dim} N$. If a map $f: K \mapsto N$ can be extended to a differentiable map $F: M \mapsto N$, then for every regular value $y \in N$, the degree satisfies
	\begin{equation*}
		\deg (f, y) = 0.
	\end{equation*}
\end{lemma}

The proof of Theorem \ref{th:2-3}:

Since $B(r)$ is compact, it follows from 2) that $E$ contains finite number of zero points. By 1), $E \subset \text{int}\, B(r)$. For each $x \in E$, let $B_x \subset B(r)$ be a small open ball centered at $x$, so that $\bar{B} = B(r) - \bigcup_{x \in E}B_x$ is a manifold with boundary
\begin{equation*}
	\partial \bar{B} = S^{n-1}(r) \bigcup \cup_{x \in E}\partial B_x.
\end{equation*}

Now consider the map $\bar{f}: \partial \bar{B} \mapsto S^{n-1}(1)$
\begin{equation*}
	\bar{f}= \frac{f(x)}{\lVert f(x) \rVert},\; x \in \partial \bar{B}.
\end{equation*}

For a regular value $y$ of $\bar{f}$, it follows from the above lemma that
\begin{equation*}
	\deg (\bar{f}, y) =0.
\end{equation*}

Since the degree of $\bar{f}$ is independent of regular values:
\begin{equation*}
	\deg (\bar{f}) = 0.
\end{equation*}

On the other hand, by the properties of degree,
\begin{align*}
	\deg (\bar{f}) &= \deg (\bar{f}) \vert_{S^{n-1}(r)} - \sum_{x \in E} \deg (\bar{f}) \vert_{\partial B_x} \\
	&= \text{ind} f_{B(r)} - \sum_{x \in E} \text{ind} f_{B_x}.
\end{align*}

Therefore
\begin{equation*}
	\text{ind} f_{B(r)} = \sum_{x \in E} \text{ind} f_{B_x}.
\end{equation*}

The minus in $- \sum_{x \in E} \text{ind} f_{B_x}$ is due to the fact that the orientation of $\partial B_x$ is opposite to that of $S^{n-1}(r)$ (see \cite{1965Topology} for a discussion).

\section{The indices of linear control systems with saturated state feedback}
For the control system of the form
\begin{equation*}
	\dot{x} = Ax + Bu, \, x \in \mathbb{R}^n,\, u \in \mathbb{R}^m,\, \lVert u \rVert \leq M, M > 0,
\end{equation*}
where $\lVert u \rVert = \max \{ u_i \}, u = (u_1,\cdots, u_m)$.

What we are interested in this paper is the following proplems. Assume that $A$ is anti-stable, i.e., every eigenvalue of $A$ has positive real part. The system $(A,B)$ is controllable.

Define sat: $\mathbb{R} \to \mathbb{R}$ as $\text{sat}(s) = \text{sign}(s) \min \{ M, \, \lvert s \rvert \}$, and for $u \in \mathbb{R}^m$,
\begin{equation*}
	\text{sat}(u) = ( \text{sat}(u_1), \, \cdots , \, \text{sat}(u_m) )^T.
\end{equation*}

By the above assumption, it is easy to see that there is stabilizing state feedback $u = Kx$, such that the closed loop system
\begin{equation}\label{eq:3-1}
	\dot{x} = Ax + B \, \text{sat}(Kx)
\end{equation}
has the origin as its asymptotically stable equilibrium.

Since $A$ is anti-stable, the attraction basin is bounded, and the boundary of the attraction basin is of much interest from both of theoretical and practical point of view. On the other hand, the locations of other (unstable) equilibrium points are also of some interest because the "size" of the attraction basin can not be large to contain the equilibrium points other than the origin!

As noted in \cite{hu2001control}, system \eqref{eq:3-1} may has $3^m$ "potential" equilibrium points. However, only some of them are true equilibrium points.

Note that $K = (k_1, \cdots, k_m)^T$, where $k_i$ is row with $n$-entries.

\begin{definition}
	Consider the equation
	\begin{equation}\label{eq:3-2}
		Ax + B \text{sat}(Kx)  = 0.
	\end{equation}
	A zero point of \eqref{eq:3-2} is said to be in general position if it is not on the plane $k_i x = \pm M$, for every $i \in \{ 1, \cdots, n \}$.
\end{definition}

Clearly, in generic case, each zero point of \eqref{eq:3-2} is in general position. For convenience, we consider the control system with single input.

\begin{theorem}\label{th:}
	For control system with single input
	\begin{equation}\label{eq:3-3}
		\dot{x} = Ax + bu,\; x \in \mathbb{R}^n,
	\end{equation}
	let $u = kx$ be a stabilizing state feedback for \eqref{eq:3-3}, then generically the system
	\begin{equation}\label{eq:3-4}
		\dot{x} = Ax + b \, \text{sat}(kx),
	\end{equation}
	has a unique equilibrium point, the origin, if $n$ is an even number, and has three equilibrium points if $n$ is an odd number.
\end{theorem}

\begin{proof}
	In generic case every equilibrium point is not on the hyperplane $kx = \pm M$. Since the index of the origin is $(-1)^n$, and the other equilibrium point has index 1, because $A$ is anti-stable and all these equilibrium points are located off the saturated region.
	
	Now for $r$ sufficently large, following from Proposition \ref{pro:2-2}, we have
	\begin{equation*}
		\textnormal{ind}(Ax + b \, \text{sat}(kx))_{B(r)} = \textnormal{ind}(Ax)_{B(r)} = 1.
	\end{equation*}
	Thus by Theorem \ref{th:2-3},
	\begin{equation*}
		\sum_{x \in E}\textnormal{ind}_f \, (x) = 1, \; f = Ax + b \, \text{sat}(kx).
	\end{equation*}
	Consequently we have the conclusions in the theorem.
\end{proof}

\section{Further discussions on properties of attraction basin boundary}
Since our concern with the equilibrium points of the closed loop stabilized system is how to characterize the boundary of the attraction basin of the origin, we will give a brief discussion on the boundary topic in this section.

In view of the differential topology theory, the following is obvious.

\begin{proposition}
	For the stabilized closed loop system \eqref{eq:3-4}, if the attraction basin of the origin is homeomorphic to $S^{n-1}(1)$, then the equilibrium points other than the origin all lie on the boundary if $n$ is odd, and no equilibrium point lies on the boundary if $n$ is even.
\end{proposition}

An interesting question is whether the boundary of the attraction basin is convex if it is homeomorphic to $S^{n-1}(1)$. It is well known that the null controllability region is convex if the input set is convex, and in the two dimensional case, it has been shown by Hu and Lin \cite{hu2001control} that attraction basin of the origin is bounded by a limit cycle. They also provided an elegant proof of convexity of the limit cycle. All these results give rise to the expectation that the boundary of attraction should be convex if it is homeomorphic to the sphere even in higher dimensional case.

Unfortunately, this is denied by the following example.

Consider the closed loop system \eqref{eq:3-4} with
\begin{equation}\label{eq:3-5}
	A =
	\begin{bmatrix}
		1 & -3 & 0 \\
		3 & 1 & 0 \\
		0 & 0 & 4
	\end{bmatrix},\;
	b =
	\begin{bmatrix}
		1 \\
		2 \\
		4
	\end{bmatrix},\;
	k = [\frac{7}{3}, -\frac{4}{3}, -\frac{35}{12}].
\end{equation}
The eigenvalues of $A + bk$ are -1,-2 and -3, hence the origin is asymptotically stable.

The attraction basin $D$ of the origin can be obtained by numerical simulation, as shown in Figure \ref{fig:1}. The boundary of $D$ is divided into two parts by a periodic orbit $\Gamma$, one of which is colored and the other is transparent. These two parts are symmetric about the origin.

\begin{figure}[!ht]
	\centering
	\includegraphics[width=0.4\textwidth]{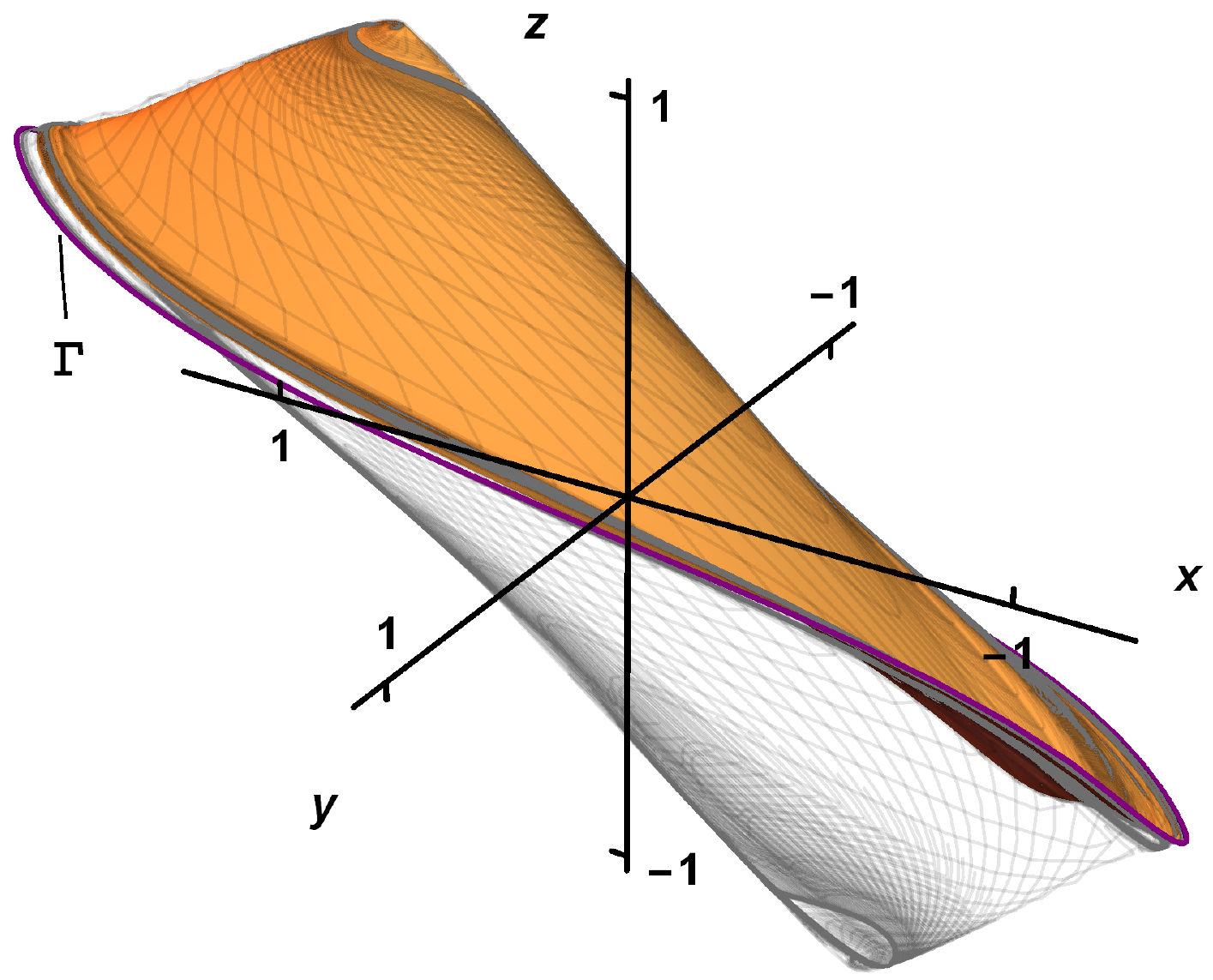}
	\caption{The boundary of the attraction basin of the origin of system \eqref{eq:3-4} with parameters \eqref{eq:3-5}.}\label{fig:1}
\end{figure}

In particular, let
\begin{align*}
	p_1 &= (-1.080860, -0.487008, -0.804244), \\
	p_2 &= (0.514148, -0.183494, 0.797384), \\
	p_3 &= (-0.283356, -0.335251, -0.003430),
\end{align*}
which $p_3$ is the midpoint of $p_1$ and $p_2$, i.e., $p_3 = (p_1 + p_3)/2$.

By numerical calculation, we find that $p_1, p_2 \in D$, but $p_3 \not\in D$. Three different trajectories starting from $p_1,p_2$ and $p_3$ respectively are shown in Figure \ref{fig:2}.

\begin{figure}[!ht]
	\centering
	\includegraphics[width=0.4\textwidth]{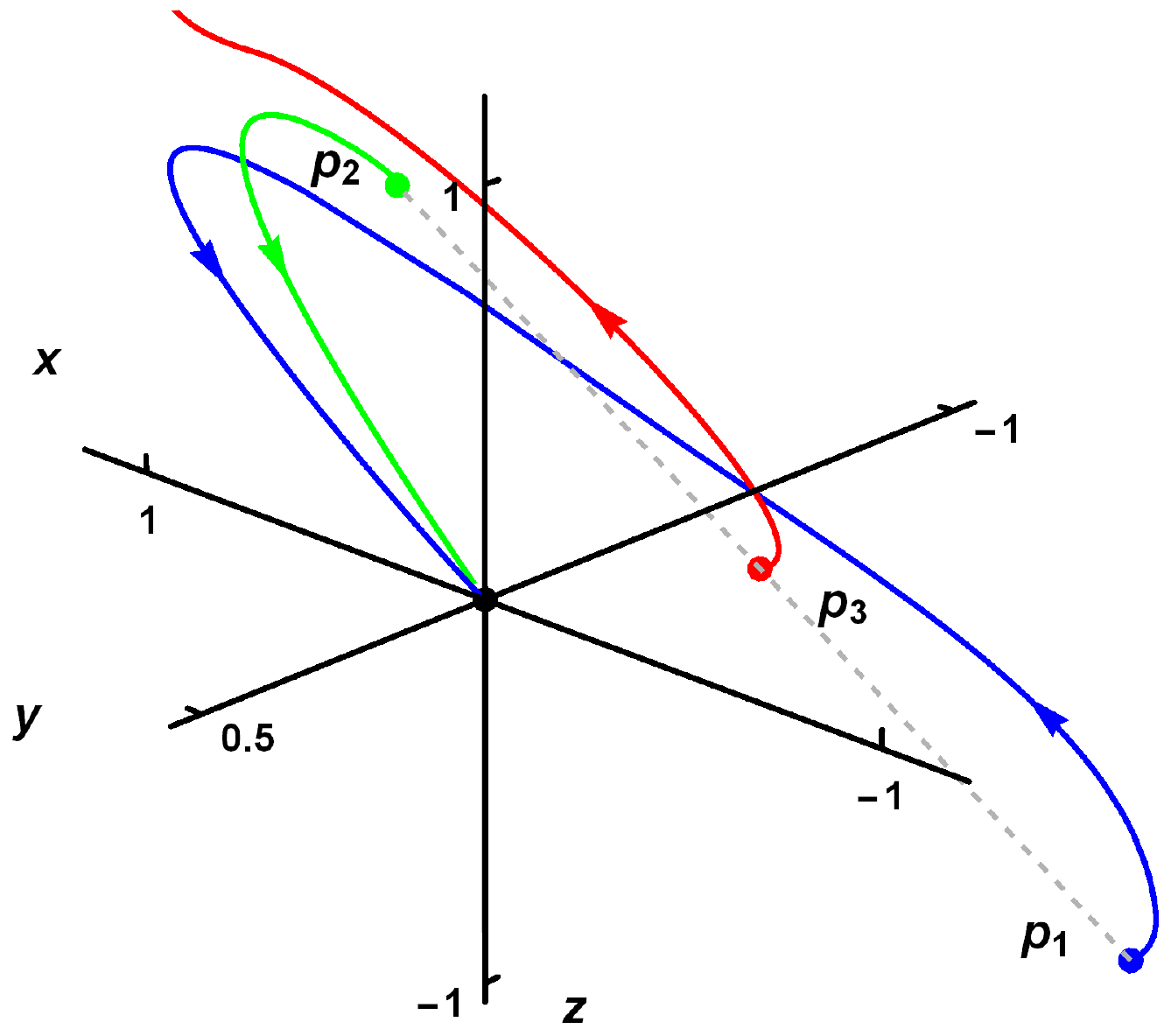}
	\caption{The trajectories of system \eqref{eq:3-4} with parameters \eqref{eq:3-5}.}\label{fig:2}
\end{figure}

This counter-example shows that the attraction basin of the origin of three-dimensional system \eqref{eq:3-2} can be non-convex, which is not possible in two-dimensional case.



\section{Acknowledgement}
This work is partially supported by National Natural Science Foundation of China (51979116).

\bibliography{ref}




\end{document}